\RequirePackage[l2tabu, orthodox]{nag}
\RequirePackage{fixltx2e}
\documentclass[10pt,a4paper,reqno,oneside,final]{smfart}

\usepackage{leftidx}
\usepackage{appendix}
\usepackage{tikz}\usetikzlibrary{decorations.markings,matrix,arrows}
\usepackage{amsfonts}
\usepackage{amsthm}
\usepackage[T1]{fontenc}
\usepackage[mathscr]{eucal}
\usepackage{setspace}
\usepackage{amssymb,latexsym,amsmath,amscd}
\usepackage{graphicx}
\usepackage{showkeys}
\usepackage[french]{babel}
\usepackage{layout}
\usepackage{enumerate}
\usepackage{indentfirst}
\usepackage[varg]{pxfonts}
\usepackage[stretch=10]{microtype}
\usepackage{booktabs}
\usepackage{xspace}
\usepackage{eufrak}
\usepackage[colorlinks=false, pdfborder={0 0 0}]{hyperref} 
\usepackage{nameref}
\usepackage{cleveref}
\usepackage{calrsfs}
\usepackage{filecontents}
\usepackage{mathtools}
\usepackage{titlesec}
\usepackage{fixltx2e}
\usepackage{todonotes}

\titleformat{\paragraph}
{\normalfont\normalsize\bfseries}{\theparagraph}{1em}{}
\titlespacing*{\paragraph}
{0pt}{3.25ex plus 1ex minus .2ex}{1.5ex plus .2ex}

\newtheoremstyle{exostyle} 
{\topsep}% espace avant 
{\topsep}% espace apres 
{}% Police utilisee par le style de thm 
{}% Indentation (vide = aucune, \parindent = indentation paragraphe) 
{\bfseries}% Police du titre de thm 
{.}% Signe de ponctuation apres le titre du thm 
{ }% Espace apres le titre du thm (\newline = linebreak) 
{\thmname{#1}\thmnumber{ #2}\thmnote{. \normalfont{\textit{#3}}}}% composants du titre du thm : \thmname = nom du thm, \thmnumber = numéro du thm, \thmnote = sous-titre du thm 

\theoremstyle{exostyle} 
 
\makeatletter
\newcommand{\neutralize}[1]{\expandafter\let\csname c@#1\endcsname\count@}
\makeatother

\theoremstyle{plain}
\newtheorem{thm}{Theorem}[section]

\newtheorem*{thm*}{Theorem}

\newtheorem{lem}[thm]{Lemma}

\newtheorem{pro}[thm]{Proposition}
\newtheorem{pro-def}[thm]{Proposition-Definition}
\newtheorem{cor}[thm]{Corollary}

\theoremstyle{definition}

\newtheorem{rem}[thm]{Remark}

\theoremstyle{remark}

\newcommand{\bP}{\mathbf{P}}

\newcommand{\bC}{\mathbf{C}}

\newcommand{\bR}{\mathbf{R}}

\newcommand{\bQ}{\mathbf{Q}}

\newcommand{\bZ}{\mathbf{Z}}

\newcommand{\ga}{\alpha}
\newcommand{\gS}{\Sigma}
\newcommand{\gb}{\beta}
\newcommand{\gk}{\kappa}

\newcommand{\go}{\omega}
\newcommand{\gs}{\sigma}

\newcommand{\cT}{\mathcal{T}}
\newcommand{\cM}{\mathcal{M}}

\newcommand{\cO}{\mathcal{O}}

\newcommand{\ol}{\overline}

\newcommand{\colonec}{\mathrel{:=}}

\newcommand{\Alb}{\mathrm{Alb}}

\newcommand{\GW}{\mathrm{GW}}

\newcommand{\NE}{\mathrm{NE}}
\newcommand{\Mbar}{\ol{\mathcal{M}}}
\newcommand{\oNE}{\ol{\mathrm{NE}}}

\newcommand{\expe}{\mathrm{exp}}
\newcommand{\vir}{\mathrm{vir}}

\newcommand{\coker}{\mathrm{coker}}
\newcommand{\Ima}{\mathrm{Im}}

\newcommand{\dr}{\partial}
\newcommand{\cupp}{\mathbin{\smile}}
\renewcommand{\(}{\left(}
\renewcommand{\)}{\right)}

\newcommand{\vast}{\bBigg@{4}}
\newcommand{\Vast}{\bBigg@{5}}

\makeatletter
\let\orgdescriptionlabel\descriptionlabel
\renewcommand*{\descriptionlabel}[1]{%
  \let\orglabel\label
  \let\label\@gobble
  \phantomsection
  \edef\@currentlabel{#1}%
  \let\label\orglabel
  \orgdescriptionlabel{#1}%
}
\makeatother
\tikzset{node distance=2cm, auto}

\numberwithin{equation}{section}

\title{Compact Kähler threefolds with non-nef canonical bundle and symplectic geometry} 

\author{Hsueh-Yung Lin}
\address{Centre de Mathématiques Laurent Schwartz, 91128 Palaiseau Cédex, France}

\begin{document}

\maketitle

\begin{abstract}
We show that the  nefness of the canonical bundle of compact Kähler threefolds is invariant under deformed symplectic diffeomorphisms.
\end{abstract}

\section{Introduction}
Consider a compact Kähler manifold $(X,\omega)$   and its canonical bundle $K_X$. We say that $K_X$ is \emph{algebraically nef} if its degree is non-negative on every curve $C$ on $X$. It should not be confused with the \emph{a priori} stronger notion of nefness in the sense of~\cite{DPSneftang}: a line bundle $L$ on $X$ is called \emph{nef} if its first Chern class $c_1(L) \in H^2(X,\bR)$ lies in the closure of the Kähler cone of $X$. These two notions agree on smooth projective varieties.

Let $(Y,\omega')$ be another compact Kähler manifold. A \emph{deformed symplectic diffeomorphism} $\phi : X \to Y$ is a diffeomorphism such that the symplectic form $\phi^*\go'$ is in the same deformation class of symplectic forms as $\go$. If such a diffeomorphism $\phi$ exists, we  say that $X$ and $Y$ are \emph{symplectically equivalent}.

A Kähler manifold equipped with its Kähler form is also a symplectic manifold, and furthermore, the set of Kähler forms is connected and thus determines a deformation class of symplectic forms. Hence, it is natural to ask which properties coming from algebraic or Kähler geometry depend only on the symplectic deformation equivalence class. This kind of questions was first studied by Ruan~\cite{Ruan1,Ruan2} by showing that the  nefness of the canonical bundle of smooth projective surfaces or threefolds is invariant under deformed symplectic diffeomorphisms. In the same direction, Kollár~\cite{Kollar} and Ruan~\cite{Ruan} showed that uniruledness of Kähler manifolds is a symplectic invariant. The same holds true for rational connectedness of smooth projective varieties of dimension up to three~\cite{Voisinrc,ZTian} and some of dimension four~\cite{ZTian1}.

The aim of this note is to prove the following theorem, which generalizes Ruan's result:

\begin{thm}\label{ch1main}
The algebraic nefness of the canonical bundle of compact Kähler threefolds is invariant under deformed symplectic diffeomorphisms: given two symplectically equivalent compact Kähler threefolds $X$ and $Y$, assume that  $K_X$ is algebraically nef, then  $K_Y$ is also algebraically nef.
\end{thm}

Combining~\cite[Corollary 4.2]{HorPet} and the fact that a Kähler threefold is \emph{not} uniruled if and only if its canonical line bundle is pseudo-effective~\cite{BrunellaPosFol}, we derive as a corollary of Theorem~\ref{ch1main} the following stronger result:

\begin{cor}\label{mainfort}
The same conclusion of Theorem~\ref{ch1main} holds if "algebraic nefness" is replaced by "nefness" in the sense of~\cite{DPSneftang}.
\end{cor}

Let $X$ be a compact Kähler manifold. Set $N_1(X)_{\bR}$
%\footnote{By Hodge theory, this coincides with the set of real Hodge classes $H^{1,1}(X,\bC) \cap H^2(X,\bR)$.} 
the real vector space of $1$-cycles modulo numerical equivalence. Inside $N_1(X)_{\bR}$ sits the \emph{effective curve cone} $$\NE(X) \colonec \left\{\sum_{\text{finite}}a_i [C_i] \ \bigg{|} \ a_i > 0, C_i \text{ irreducible curves of } X\right\} \subset N_1(X)_{\bR},$$
the set of classes of effective $1$-cycles. We denote by $\oNE(X)$ the closure of $\NE(X)$ inside $N_1(X)_{\bR}$.

More generally, given a symplectic manifold $(M,\go)$ and an $\go$-tamed almost complex structure $J$ on $M$, we define
$$\NE(M)_{\go,J} \colonec \left\{\sum_{\text{finite}}a_i [C_i] \ \bigg{|} \ a_i > 0, C_i \text{ $J$-holomorphic curves of } M\right\} \subset H_2(M,\bR).$$
In~\cite{Ruan1}, Ruan defined the \emph{deformed symplectic effective cone} as the intersection for all symplectic forms $\go'$ which can  deform to $\go$ and all $\go'$-tamed almost complex structures $J$
$$\mathrm{DNE}(M) \colonec \bigcap_{\go'} \bigcap_{J} \NE(M)_{\go,J}.$$ 
In the same paper, he stated the following \emph{main criterion}: 

\begin{thm}\label{crit}
Suppose that $[C]$ is extremal\footnote{An element $x$ in a closed cone $C$ is called \emph{extremal} if whenever $x = y + z$ for $y,z \in C$, then $y$ and $z$ are multiples of $x$.} in $\oNE(M)_{\go',J}$ for some $\go'$ and $J$. If the Gromov-Witten invariant $\GW_{[C]}$ is non-zero, then $[C]$ is extremal in $\ol{\mathrm{DNE}}(M)$.
\end{thm}

We refer to Section~\ref{GW} for a brief reminder of Gromov-Witten invariants.

When $X$ is a smooth projective threefold, Mori gave a classification of the extremal rays in
$$\oNE (X)_{K_X < 0} \colonec \ol{\left\{ [C] \in \NE(X) \ \mid \ (K_X,C)<0 \right\} }$$
and showed that they are all generated by rational curves. We define the \emph{minimal homology class} in an extremal ray $R$ the class of a rational curve $[C]\in R$ such that $(-K_X, C)$ is minimal.  Ruan used his main criterion to prove that extremal curves do not disappear under symplectic deformation for projective threefolds by verifying that the minimal homology class $[C]$ in each extremal ray of $\oNE (X)_{K_X < 0} $  has a non-zero genus zero Gromov-Witten invariant $\GW_{0,[C]}$. Since the canonical class is also a symplectic deformation invariant, \emph{i.e.} $\phi^*c_1(K_Y) = c_1(K_X)$ (indeed, $c_1(K_X)$ is determined by any $\omega$-tamed  almost complex structure which form a connected set) and since genus zero Gromov-Witten invariants can be calculated by doing intersection theory on the moduli space of rational curves~\cite{LiTian}~\cite{Siebert}, we conclude that there exists a rational curve $C' \subset Y$ such that $\phi^*[C'] = [C]$ and $(K_Y, C') = (\phi^*K_Y, \phi^*[C']) = (K_X,C)<0$.

Back to the Kähler case, since the projective case was verified by Ruan, it suffices to prove  Theorem~\ref{ch1main} for non-algebraic compact Kähler threefolds. Recently, Höring and Peternell generalized Mori's cone theorem  for compact Kähler threefolds~\cite{HorPet}.
\begin{thm}[Höring, Peternell]\label{mori}
Let $X$ be a compact Kähler threefold. There exists a countable set $(C_i)_{i \in I}$ of rational curves on $X$ such that
$$0 < -K_X \cdot C_i \le \dim(X) + 1$$
and
$$\oNE(X ) = \oNE (X)_{K_X \ge 0} + \sum_{i\in I}\bR^+[C_i]$$
where $\bR^+[C_i]$ are distinct extremal rays of $\oNE(X)$ contained in $N_1(X)_{K_X < 0}$. These rays are locally discrete in that half-space. 
\end{thm}
Theorem~\ref{mori} together with the classification list of non-splitting families of rational curves on Kähler threefolds given by Campana and Peternell~\cite{Peter1} allow us to follow  Ruan's method to prove Theorem~\ref{ch1main} in the Kähler context.

This note is organized as follows. In Section~\ref{classi}, we exhibit the classification list of non-splitting families of rational curves on Kähler threefolds. In Section~\ref{GW} we recall some elementary properties of Gromov-Witten invariants which we will use. Finally, we prove Theorem~\ref{ch1main} and derive from it Corollary~\ref{mainfort} in the last section.  

If $Z$ is an analytic (sub)space, $[Z]$ will denote interchangeably its fundamental homology class and its Poincaré dual cohomology class throughout this note.

%Je remercie ma directrice de thèse Claire Voisin pour m'avoir partagé ses idées et ses remarques précieuses et éclairantes. 

\section{Classification of non-splitting families of rational curves}\label{classi}

Let $X$ be a compact Kähler threefold. A  family of rational curves $(C_t)_{t\in T}$ in $X$ is called \emph{non-splitting}, if $T$ is compact and irreducible and every curve $C_t$ is irreducible. Every extremal rational curve $C$ such that $(K_X,C) < 0$ determines a non-splitting family of rational curves $(C_t)_{t\in T}$~\cite{Peter1}, and $\dim T = -(K_X,C)$. One can classify these families according to $-(K_X,C) \in \{1,2,3,4\}$ and this was done by Campana and Peternell~\cite{Peter1}. Since we are mainly interested in non-algebraic compact threefolds, we shall exclude projective varieties from the classification list.

\begin{thm}[Campana-Peternell]\label{class}
Let $X$ be a non-algebraic compact Kähler threefold and $(C_t)$ a non-splitting family of rational curves. Then either $(K_X,C_t) = -2$ or $-1$. Moreover,
\begin{enumerate}
\item If $(K_X,C_t) = -2$, then
\begin{enumerate}
\item if  $C_t$ fills up a surface $S \subset X$, then $S \simeq \bP^2$ with normal bundle $N_{S/X} = \cO(-1)$ and $(C_t)$ is the family of lines;
\item if $C_t$ fills up $X$, then  $X$ is a $\bP^1$-bundle over a surface, the $C_t$ being fibers.
\end{enumerate}
\item If $(K_X,C_t) = -1$, then $C_t$ fills up a surface $S \subset X$. 
\begin{enumerate}
\item If $S$ is normal, then one of the following holds:

\begin{enumerate}
\item $S\simeq \bP^2$ with $N_{S/X} = \cO(-2)$ and $(C_t)$ is a family of lines;
\item $S\simeq \bP^1 \times \bP^1$ with $N_{S/X} = \cO(-1) \boxtimes \cO(-1)$ and the $C_t$ are lines in $S$;
\item $S$ is a quadric cone with $N_{S/X} = \cO(-1)$;
\item $S$ is a ruled surface over a smooth curve $B$, the $C_t$ being fibers of $\pi : S \to B$, and $X$ is the blow-up of a smooth threefold along a section of $\pi$.
\end{enumerate}
\item If $S$ is non-normal, then $\gk(X)<0$ and $N_{S/X} = \cO_S$. The normalization $\tilde{S}$ of $S$ is either $\bP^2$ or a ruled surface $\pi:\tilde{S} \to B$. In the formal case, the pre-image $\tilde{C}_t$ of $C_t$ under the normalization map $\tilde{S} \to S$ is a line in $\bP^2$. In the latter case, the ruling of $\tilde{S}$ can be chosen so that $\tilde{C}_t$  is a fiber of $\pi$. 
\end{enumerate}

\end{enumerate}

\end{thm}

\section{Gromov-Witten invariants}\label{GW}

We refer to~\cite{McD} for more details concerning Gromov-Witten theory.
Let $(X,\go)$ be a compact symplectic manifold endowed with an $\go$-tamed almost complex structure $J$.
Let $A \in H_2(X,\bZ)$. We denote by $\Mbar_{A,n}(X)$ (or $\Mbar_{A,n}$ if there is no ambiguity) the moduli space of stable maps from curves of genus $0$ to $X$ with $n$ marked points, whose homology class of its image is equal to $A$. This is a compactification of the moduli space $\cM_{A,n}$ of maps $u : C = \bP^1 \to X$ with $n$ marked points such that $u_*[C] = A$. When there is no marked point, this moduli space is simply denoted by $\Mbar_{A}$. Since $X$ is a complex threefold, we recall that the \emph{expected} (or virtual) dimension of $\Mbar_{A,n}$
$$d_{\expe} = n-\int_A c_1(K_X),$$ 
and  $\Mbar_{A,n}$ carries  a virtual fundamental class of expected dimension $2d_{\expe}$ 
$$[\Mbar_{A,n}]^{\vir} \in H_{2d_{\expe}}(\Mbar_{A,n},\bQ).$$
Gromov-Witten invariants are defined by capping the cohomology classes against the virtual fundamental class of the space of stable maps. More precisely, given cohomology classes $A_1,\ldots,A_n$ in $H^*(X,\bQ)$, the corresponding genus zero Gromov-Witten invariant is defined by:
$$GW_{0,A}(A_1,\ldots,A_n) \colonec \int_{[\Mbar_{A,n}]^{\vir}} e_1^*(A_1) \cupp \cdots \cupp e_n^*(A_n),$$
where $e_i$ denotes the evaluation map with respect to the $i$th marked point
$$\begin{aligned}
 e_i : \Mbar_{A,n} & \longrightarrow X \\
\(\gS;p_1,\ldots,p_n;f\) & \longmapsto f(p_i).
\end{aligned}$$

 When there is no obstruction on $\Mbar_{A,n}$, virtual fundamental class coincides with ordinary fundamental class. For instance, this happens when $H^1(C,f^*TX) = 0$ for all stable maps $f:C \to X$. When $\Mbar_{A,n} = \cM_{A,n}$ and is non-singular, the canonical obstruction sheaf $\cT^2 \colonec (\coker {D})_{|\cM_{A,n}}$ is a vector bundle and the virtual fundamental class is the Euler class of $\cT^2$:

$$[\Mbar_{A,n}]^{\vir} = e(\cT^2) \cap [\Mbar_{A,n}].$$
%see~\cite{Mirror} for more details.
Here $D$ is a map between (infinite dimensional) vector bundles over the space of smooth maps $u : \bP^1 \to X$ representing $A$, which is the vertical differential of the map $u \mapsto (u,\bar{\dr}_J u)$.

For each embedded curve $f : C \to X$, the pushforward under $f$ of the fundamental class of $C$ and its Poincaré dual in $H^{\dim_\bR X-2}(X,\bQ)$ are all denoted by $[C]$ in this note.
 
\begin{rem}\label{rem}
%This remark is due to Claire Voisin.
Theorem~\ref{ch1main} is true for compact Kähler surfaces for simple reasons. Let $X$ and $Y$ be compact Kähler surfaces. If $K_X$ is not algebraically nef, then either $X$ contains a $(-1)$-curve $C$, or $\gk(X) < 0$. If $f : C \to X$ is an embedded $(-1)$-curve, then $H^1(C,f^*TX) = H^1(\bP^1,\cO(-1)) = 0$, so there is no obstruction on $\Mbar_{[C]}$. The moduli space is zero dimensional, and $\GW_{0,[C]}=1$. If $\gk(X)<0$, then by classification of minimal compact Kähler surfaces, $X$ is uniruled. Since uniruledness is preserved by deformed symplectic diffeomorphism\cite{Kollar,Ruan}, $Y$ is also uniruled. In either case, $K_Y$ is not algebraically nef.
\end{rem}

\section{Non-nef canonical bundle and symplectic geometry}

This section is devoted to the proof of  Theorem~\ref{ch1main}.

\begin{proof}[Proof of Theorem~\ref{ch1main}]

Suppose $K_X$ is not algebraically nef; let $C$ be an extremal rational curve in Höring-Peternell's cone Theorem~\ref{mori}. The case where $X$ is  projective was treated in~\cite{Ruan}. As before, we thus assume that $X$ is non-algebraic. 
The curve $C$ determines a non-splitting family of rational curves $(C_t)$ which is classified in Theorem~\ref{class}. Since the canonical class $c_1(K_X)$ is a symplectic invariant, it suffices to show that some genus zero Gromov-Witten invariants $GW_{0,[C]}(\cdots)$ is non-zero for all cases listed in Theorem~\ref{class}.

\begin{lem}\label{nonzero}
For the cases $(1.a)$, $(1.b)$, $(2.a.ii)$, $(2.a.iii)$ and $(2.a.iv)$, the moduli space $\Mbar_{[C]}$ is unobstructed. Furthermore, $\GW_{0,[C]} \ne 0$.
\end{lem}

\begin{proof}[Proof]

Since the normal bundles $N_{C|X}$ in the cases $(1.a)$, $(1.b)$, $(2.a.ii)$, $(2.a.iii)$ and $(2.a.iv)$ are $$\cO(1) \oplus \cO(-1), \cO \oplus \cO, \cO \oplus \cO(-1), \cO \oplus \cO(-1), \text{ and } \cO \oplus \cO(-1)$$ respectively, one has $H^1(C,f^*T_X) = H^1(C,N_{C|X}) = 0$, so the moduli space  $\Mbar_{[C]}$ in the cases above is unobstructed.

For $(1.a)$, let 
$$e \colonec (e_1,e_2) : \Mbar_{[C],2} \to X \times X$$
where $e_i$ is the evaluation map with respect to the $i$th marked point. Let $S \simeq \bP^2 \subset X$ denote the image of $e :  \Mbar_{[C],1} \to X $,  then
$$\GW_{0,[C]}\([C],[C]\) = \int_{\Mbar_{[C],2} } e^*[C \times C] = [C \times C] \cdot [S \times S] = p_1^*\([C] \cdot [S]\) \cdot p_2^*\([C] \cdot [S]\) = 1$$
where $p_i$ is the projection $X \times X \to X$ onto the $i$th component.
%For $(1.a)$, since $N_{S|X} = \cO(-1)$, we can perturb $C$ to get a genus $0$ sphere $C'$ such that $C'$ meets transversally $S$ at a point $x$ with intersection number $(-1)$. Hence $\GW_{0,[C]}\([C],[C]\) =1$. 

For $(1.b)$, it is  clear that $\GW_{0,[C]}\([x]\) =1$ with $x\in X$. In the cases $(2.a.ii)$, $(2.a.iii)$ and $(2.a.iv)$, one has $\GW_{0,[C]}\([C]\) =[C] \cdot [S] = -1$ by the projection formula. 
\end{proof}

\begin{rem}
Another way to show that $\GW_{0,[C]}\([C],[C]\) = 1$ in the case $(1.a)$ is by perturbing (non-algebraically) a line $l \subset S$ out of the surface $S$ to get two lines $l'$ and $l''$ intersecting $S$ transversally and negatively in two different points. We refer to the proof of Proposition $5.6$ (case "Type $E_2$") in~\cite{Ruan} for the detail.
\end{rem}

\begin{rem}
We can also consider Lemma~\ref{nonzero} for the cases $(1.b)$, $(2.a.ii)$, $(2.a.iii)$ and $(2.a.iv)$ as a consequence of the following Lemma, which slightly generalizes Lemma $5.3$ in~\cite{Ruan} in the context of Kähler geometry:
\end{rem}

\begin{lem}\label{nonzero1}
Suppose that the deformation of $C$ is unobstructed. Let $e : \Mbar_{[C],1} \to X$ be the evaluation map. If $\dim \Mbar_{[C]} \le 2$ and 
$$\dim e(\Mbar_{[C],1} ) = \dim \Mbar_{[C],1}, $$ then $\GW_{0,[C]} \ne 0$.
\end{lem}

\begin{proof}[Proof]
Since $\dim e(\Mbar_{[C],1} ) = \dim \Mbar_{[C],1} $, one has 
$$ e_*[\Mbar_{[C],1}] = k[e(\Mbar_{[C],1} )] $$
where $k = \deg e \ne 0$.

If $\dim \Mbar_{[C]} =2$ and $x\in X$ is a point in general position, then
$$\GW_{0,[C]}(x) = e^*[x] \cupp [\Mbar_{[C],1}] = k \ne 0$$ by projection formula.

If $\dim \Mbar_{[C]} =1$, again by projection formula, one has 
$$\GW_{0,[C]}(\ga) = k \int_{e\(\Mbar_{[C],1} \)} \ga,$$
for $\ga \in H^4(X,\bQ)$. Since $H^4(X,\bQ)$ is dense in $H^4(X,\bR)$, we can choose $\ga \in H^4(X,\bQ)$ sufficiently close to  $\go \wedge \go$, hence 
$$\GW_{0,[C]}(\ga) = k \int_{e\(\Mbar_{[C],1} \)} \ga \ne 0.$$
\end{proof}

For $(2.a.i)$, one has $H^1(C,f^*T_X) = H^1(C,N_{C|X}) = \bC$, so the deformation of $C$ is obstructed. Since $\Mbar_{[C]} \simeq G(2,3) \simeq \bP^2$ is non-singular, the virtual fundamental class $$[\Mbar_{A}]^{\vir} = e(\cT^2) \cap [\Mbar_{A}],$$
where $e(\cT^2)$ is the Euler class of the obstruction bundle $\cT^2$ and was first computed by Ruan~\cite{Ruan1}.

\begin{lem}
$e(\cT^2) = -\gs_1$ where $\gs_1$ is the Schubert cycle which represent all the lines in $\bP^2$ passing through a point in general position. Moreover, $\GW_{0,[C]}\([C]\) = [C] \cdot [S] = -2$.
\end{lem}

\begin{proof}[Proof]
\emph{cf.}~\cite[Section 5]{Ruan1}. 
\end{proof}

It remains the case $(2.b)$, where $C_t$ fills up a non-normal surface $S$. We denote by $\nu : \tilde{S} \to S$ the normalization of $S$. 

\begin{rem}\label{conjuniruled}
Using the abundance conjecture, proven for non-simple Kähler threefolds by Peternell~\cite{Peter3} and for all Kähler threefolds by Campana, Höring and Peternell~\cite{CHPabun}, the remaining case can be settled easily. Indeed, threefolds in case $(2.b)$ are uniruled since they have negative Kodaira dimension. Since uniruledness is  symplectic invariant~\cite{Kollar,Ruan}, we conclude that $Y$ is also uniruled for any symplectic deformation (Kähler threefold) $Y$ of $X$. Hence $K_Y$ is not algebraically nef. However, below we will provide a more direct proof without using the abundance conjecture. 
\end{rem}

\begin{lem}
If $\pi: \tilde{S} \to B$ is a ruled surface over a smooth curve $B$, then $\GW_{0,[C]} \ne 0$.
\end{lem}

\begin{proof}[Proof]
First we note that $\dim \Mbar_{[C]} \ge -(K_X,C) = 1$. If $\dim \Mbar_{[C]}  > 1$,  then $(C_t)$ will fill up $X$ because $\dim \Mbar_{[C]}(S) = 1$ by hypothesis. Accordingly, $\GW_{0,[C]}(x) \ne 0$ where $x$ is the class of any point in $X$. If $\dim \Mbar_{[C]}  = 1$, then the deformation of $C$ is unobstructed; we can therefore conclude by Lemma~\ref{nonzero1}.

%Since $\nu$ is finite and $\Mbar_{[C]}(X) = \Mbar_{[C]}(S)$, the normalization of $\Mbar_{[C]}(X)$ is included in the moduli space $\Mbar_{\nu^*[C]}(\tilde{S})$. It follows that $\dim \Mbar_{[C]} \le \Mbar_{\nu^*[C]}(\tilde{S}) \le 1$ because $\tilde{S}$ is a ruled surface. 
 %As in the proof of Lemma~\ref{nonzero1}, since $\dim e(\Mbar_{[C],1} ) = \dim S = 2 = \dim \Mbar_{[C],1}$ and $\dim [\Mbar_{[C],1}] = \dim [\Mbar_{[C],1}]^{\text{vir}}$, one has
%$$e_*[\Mbar_{[C],1}]^{\text{vir}} = k[e(\Mbar_{[C],1})]$$
%for some non-zero integer $k$. Hence $GW_{0,[C]}(\ga) \ne 0$ for $\ga \in H^4(X,\bQ)$ sufficiently close to $\go \wedge \go$.

\end{proof}

From now on, we concentrate on the last remaining case in $(2.b)$, that is the case where the normalization of $S$ is $\bP^2$. Theorem~\ref{ch1main} in this case will be a direct consequence of the following
\begin{pro}\label{pro-uniruled}
In the case $(2.b)$, if the normalization of $S$ is $\bP^2$, then  $X$ is uniruled.
\end{pro}

Admitting Proposition~\ref{pro-uniruled} for the moment, we finish the proof of Theorem~\ref{ch1main} as follows: since uniruledness is  symplectic invariant~\cite{Kollar,Ruan}, we conclude that $Y$ is also uniruled so $K_Y$ is not algebraically nef.

%\begin{rem}\label{conjuniruled}
%Proposition~\ref{pro-uniruled} is in accordance with the well-known conjecture that if $X$ is a compact Kähler manifold with negative Kodaira dimension, then $X$ is uniruled. For non-algebraic compact Kähler threefolds with the possible exception of simple manifolds (that is, manifolds such that there is no non-trivial subvariety passing through a general point), this conjecture is known to be true by~\cite{Peter3}, which is a consequence of the abundance conjecture for non-simple Kähler threefolds.
%\end{rem}

\begin{proof}[Proof of Proposition~\ref{pro-uniruled}]
We will first prove the following lemma:
\begin{lem}\label{lem1}
$X$ is a fiber space over a smooth curve and $S$ is contracted to a point.
\end{lem}

\begin{proof}[Proof]
Assume that $H^1(X, \cO_X) = 0$. By considering the following short exact sequence:
\begin{center}
\begin{tikzpicture}[description/.style={fill=white,inner sep=2pt}]
\matrix (m) [matrix of math nodes, row sep=3em,
column sep=2.5em, text height=1.5ex, text depth=0.25ex]
{ 0  & H^0\(X, \cO_X\) & H^0\(X, \cO_X(S)\) & H^0\(X,\cO_S(S)\) & 0 ,\\
 };
\path[->,font=\scriptsize]
(m-1-1) edge node[auto]{} (m-1-2)
(m-1-2) edge node[auto]{} (m-1-3)
(m-1-3) edge node[auto]{} (m-1-4)
(m-1-4) edge node[auto]{} (m-1-5);
%(m-1-1) edge node[auto] {$ \phi $} (m-1-3)
%edge node[description] {$ \pi $} (m-2-2)
%(m-1-3) edge node[above=4pt,left=0pt] {$ pr_2 $} (m-2-2);
\end{tikzpicture}
\end{center}
it is clear that $\dim H^0\(X, \cO_X(S)\) \ge 2$. The morphism $X \to \bP^1$ defined by the base-point-free linear system $| \cO_X(S) |$ does the work.

If $H^1(X, \cO_X) \ne 0$, then the Albanese map $\ga : X \to \Alb(X)$ is non-constant. If $\ga$ is generically finite, then the pullback of a general holomorphic $3$-form is non-zero in $H^0(X,K_X)$, which contradicts the fact that $\gk(X) <0$.

Now we assume that $\dim \Ima \ga \le 2$. Let $\go_0$ be a Kähler form on $\Alb(X)$ and  $Q(\beta,\gamma) \colonec \int_X\gb \wedge \gamma \wedge \go$ be the intersection form on $H^{1,1}_{\bR}(X) \colonec H^{1,1}(X,\bC) \cap H^2(X,\bR)$ determined by $\go$. The signature of $Q$ is $(1,\dim H^{1,1}_{\bR}(X)-1)$ by Hodge index Theorem~\cite[Theorem $6.2.3$]{VoisinI}. Note that since there is no non-constant map from $\bP^2 = \tilde{S} $ to any torus, $\ga$ contracts $S$ to a point, so $Q(\ga^*\go_0,[S]) = 0$. Since $\dim \Ima \ga \le 2$, one has $Q(\ga^*\go_0,\ga^*\go_0) \ge 0$, it follows that $Q$ is negative on the  orthogonal complement in $H^{1,1}_{\bR}(X)$ of the line spanned by $[\ga^* \go_0]$. We then deduce from the fact that $Q(\ga^*\go_0,[S]) = 0$ and $\cO_S(S) = \cO_S$ (so $Q([S],[S]) = 0$)   that $[S]$ is proportional to $[\ga^*\go_0]$. Hence $\dim \Ima \ga = 1$. Finally by~\cite[Proposition I$.13.9$]{Barth}, $\ga(X)$ is smooth.
\end{proof}

By Lemma~\ref{lem1},   $X$  is a fiber space over a smooth curve $\pi : X \to B$ such that $S$ is an irreducible component of a fiber $F$.  Since $Q([S],[S]) = Q([S],[F]) = 0$, $S$ is in fact a connected component of $F$. Assume that $F$ is not connected. Given an irreducible component $S'$ of $F$ disjoint with $S$, since $Q([S],[S]) = 0$ and $Q([F],[F]) = 0$, one would have $Q([S],[F]) \ne 0$ (otherwise $[S]$ would be proportional to ${F}$ again by Hodge index Theorem). Thus $Q([S],[S']) \ne 0$, which yields a contradiction. Hence the  fibers of $\pi$ are connected. In particular, we obtain $S \cdot C = 0$

%Hence $S$ is a fiber of $\pi$.

Now since $S$ is a Gorenstein surface, by~\cite[$3.34.1$]{Mori} (stated in the algebraic case, but the same proof works in the analytic case) one has $K_{\tilde{S}}  = \nu^*K_S - E$ where $E$ is the pre-image of the non-normal locus of $S$ under $\nu$. Let $\tilde{C} \subset \tilde{S}$ be the strict transform of $C$. Recall that $K_X \cdot C = -1$ and $N_{S/X} \simeq \cO_S(S)$, we obtain
$$-3 = \deg_{\tilde{C}} \(K_{\tilde{S}}\)  = K_X \cdot C + S\cdot C - \tilde{E} \cdot \tilde{C} = -1 - \tilde{E} \cdot \tilde{C}.$$
Hence $\tilde{E} \cdot \tilde{C} = 2$. Accordingly, $\nu^*K_S= \cO_{\bP^2}(-1)$, since $\tilde{C}$ is a line in $\tilde{S}$.

Therefore, one has  $H^0(S,nK_S) = 0$ for all $n > 0$, so $\gk(S) = -\infty$. The same result extends by semi-continuity to all fibers of points lie in a Zariski neighborhood of $\pi(S)$, \emph{i.e.}, there is a non-empty Zariski open $U \subset B$ such that $S_t$ is smooth and $\gk(S_t) = -\infty$  for all $t \in U$ where $S_t \colonec \pi^{-1}(t)$. These $S_t$ are all uniruled, so $X$ is also uniruled. 
\end{proof}

\end{proof}

\begin{rem}
We could have concluded the proof of Proposition~\ref{pro-uniruled} by the result of Peternell~\cite{Peter3} mentioned in Remark~\ref{conjuniruled} once we know that the Albanese map $\ga$ is not generically finite. Indeed, since $X \to B$ is fibered over a base $B$ of dimension strictly between $0$ and $\dim(X)$, $X$ is not a simple manifold. Since $X$ has negative Kodaira dimension, $X$ is uniruled.
\end{rem}

Finally, we terminate this note by a proof of Corollary~\ref{mainfort}.

\begin{proof}[Proof of Corollary~\ref{mainfort}]

Since uniruledness is invariant under symplectic deformations, by~\cite{MiyaokaMori} and~\cite{BrunellaPosFol} we see that for compact Kähler threefolds, the pseudo-effectivity of the canonical bundle is also invariant under symplectic deformations. 

Now let $X$ be a compact Kähler threefold and let $Y$ be another compact Kähler threefold which is symplectically deformed equivalent to $X$ with respect to their Kähler form. Suppose that $K_X$ is nef in the sense of~\cite{DPSneftang}, then $K_X$ is pseudo-effective and algebraically nef. Accordingly $K_Y$ is algebraically nef by Theorem~\ref{ch1main} and also pseudo-effective. Applying~\cite[Corollary 4.2]{HorPet} shows that $K_Y$ is not nef in the sense of~\cite{DPSneftang}.

\end{proof}

\bibliographystyle{alpha}
\bibliography{KXnef}

\begin{thebibliography}{BHPV04}

\bibitem[BHPV04]{Barth}
W.~Barth, K.~Hulek, C.~Peters, and A.~Van~De Ven.
\newblock {\em Compact complex surfaces}, volume~4 of {\em Ergebnisse der
  Mathematik und ihrer Grenzgebiete}.
\newblock Springer, 2 edition, 2004.

\bibitem[Bru06]{BrunellaPosFol}
Marco Brunella.
\newblock A positivity property for foliations on compact {K}\"ahler manifolds.
\newblock {\em Internat. J. Math.}, 17(1):35--43, 2006.

\bibitem[CHP14]{CHPabun}
F.~Campana, A.~Höring, and T.~Peternell.
\newblock Abundance for kähler threefolds.
\newblock {\em À paraître dans Annales de l'ENS (arXiv:1403.3175)}, 2014.

\bibitem[CP97]{Peter1}
F.~Campana and T.~Peternell.
\newblock Towards a {M}ori theory on compact {K}ähler threefolds {I}.
\newblock {\em Math. Nachr.}, 187:29--59, 1997.

\bibitem[DPS94]{DPSneftang}
Jean-Pierre Demailly, Thomas Peternell, and Michael Schneider.
\newblock Compact complex manifolds with numerically effective tangent bundles.
\newblock {\em J. Algebraic Geom.}, 3(2):295--345, 1994.

\bibitem[HP13]{HorPet}
A.~Höring and T.~Peternell.
\newblock Minimal models for kähler threefolds.
\newblock {\em À paraître dans Inventiones Math. (arXiv:1304.4013)}, 2013.

\bibitem[Kol98]{Kollar}
J.~Kollár.
\newblock Low degree polynomial equations: arithmetic, geometry and topology.
\newblock In {\em European Congress od Mathematics, Vol. I, (Budapest, 1996)},
  volume 168 of {\em Progr. Math.}, pages 255--288. Birkhäuser, 1998.

\bibitem[LT98]{LiTian}
J.~Li and G.~Tian.
\newblock Virtual moduli cycles and gromov-witten invariants of algebraic
  varieties.
\newblock {\em J. Amer. Math. Soc.}, 11:119--174, 1998.

\bibitem[MM86]{MiyaokaMori}
Yoichi Miyaoka and Shigefumi Mori.
\newblock A numerical criterion for uniruledness.
\newblock {\em Ann. of Math. (2)}, 124(1):65--69, 1986.

\bibitem[Mor82]{Mori}
S.~Mori.
\newblock Threefolds whose canonical bundles are not numerically effective.
\newblock {\em Ann. Math.}, 116:133--176, 1982.

\bibitem[MS12]{McD}
Dusa McDuff and Dietmar Salamon.
\newblock {\em {$J$}-holomorphic curves and symplectic topology}, volume~52 of
  {\em American Mathematical Society Colloquium Publications}.
\newblock American Mathematical Society, Providence, RI, second edition, 2012.

\bibitem[Pet01]{Peter3}
T.~Peternell.
\newblock Towards a {M}ori theory on compact {K}ähler threefolds {I}{I}{I}.
\newblock {\em Bull. Soc. math. France}, 129(3):339--356, 2001.

\bibitem[Rua93]{Ruan1}
Y.~Ruan.
\newblock Symplectic topology and extremal rays.
\newblock {\em Geom. Funct. Anal.}, 3(4):395--430, 1993.

\bibitem[Rua94]{Ruan2}
Y.~Ruan.
\newblock Symplectic topology on algebraic $3$-folds.
\newblock {\em J. Diff. Geometry}, 39:215--227, 1994.

\bibitem[Rua99]{Ruan}
Y.~Ruan.
\newblock Virtual neighborhood and pseudo-holomorphic curves.
\newblock In {\em Proceedings of $6$th Gökova Geometry-Topology Conference},
  volume~23, pages 161--231, 1999.

\bibitem[Sie99]{Siebert}
B.~Siebert.
\newblock Algebraic and symplectic gromov-witten invariants coincide.
\newblock {\em Ann. Inst. Fourier}, 49:1743--1795, 1999.

\bibitem[Tia12]{ZTian}
Z.~Tian.
\newblock Symplectic geometry of rationally connected threefolds.
\newblock {\em Duke Math. Journal}, 161(5):803--843, 2012.

\bibitem[Tia15]{ZTian1}
Z.~Tian.
\newblock Symplectic geometry and rationally connected 4-folds.
\newblock {\em J. Reine Angew. Math.}, 698:221--244, 2015.

\bibitem[Voi02]{VoisinI}
C.~Voisin.
\newblock {\em Hodge Theory and Complex Algebraic Geometry I}, volume~76 of
  {\em Cambridge Studies in Advanced Mathematics}.
\newblock Cambridge University Press, 2002.

\bibitem[Voi08]{Voisinrc}
C.~Voisin.
\newblock Rationally connected $3$-folds and symplectic geometry.
\newblock In {\em Géométrie différentielle, physique mathématique,
  mathématiques et société {I}{I}}, volume 322 of {\em Astérisque}, pages
  1--21, 2008.

\end{thebibliography}

\end{document}